\input ./style/arxiv-vmsta.cfg
\documentclass[numbers,compress,v1.0.1]{vmsta}
\usepackage{vtexbibtags}

\volume{5}
\issue{2}
\pubyear{2018}
\firstpage{191}
\lastpage{206}
\aid{VMSTA102}
\doi{10.15559/18-VMSTA102}
\articletype{research-article}

\ifdefined\HCode 
\def\mmmprime{\textbf{\ensuremath{'}}}
\else 
\def\mmmprime{\boldmath{\ensuremath{'}}}
\fi

\startlocaldefs
\newcommand{\lleft}{\left}
\newcommand{\angler}{\rangle}
\newcommand{\anglel}{\langle}
\newcommand{\rright}{\right}
\newcommand{\rrVert}{\Vert}
\newcommand{\llVert}{\Vert}
\newcommand{\rrvert}{\vert}
\newcommand{\llvert}{\vert}
\allowdisplaybreaks

\newtheorem{thm}{Theorem}

\newtheorem{lemma}{Lemma}
\newtheorem{cor}{Corollary}
\newtheorem*{core}{Corollary 3.4\mmmprime}

\theoremstyle{definition}
\newtheorem{defin}{Definition}
\newtheorem{exm}{Example}[section]

\endlocaldefs
\DeclareMathOperator{\diag}{diag}
\DeclareMathOperator{\Tr}{Tr}

\begin{document}

\begin{frontmatter}
\pretitle{Research Article}

\title{Large deviations of regression parameter estimator in
continuous-time models with sub-Gaussian noise}

\author{\inits{A.V.}\fnms{Alexander V.}~\snm{Ivanov}\ead[label=e1]{alexntuu@gmail.com}\orcid{0000-0001-5250-6781}}
\author{\inits{I.V.}\fnms{Igor V.}~\snm{Orlovskyi}\thanksref{cor1}\ead[label=e2]{orlovskyi@matan.kpi.ua}\orcid{0000-0003-0922-1611}}
\thankstext[type=corresp,id=cor1]{Corresponding author.}
\address{\institution{National Technical University of Ukraine}, <<Igor
Sikorsky Kyiv Polytechnic Institute>>,
Department of Mathematical Analysis and Probability Theory, Peremogi
avenue 37, 03056 Kiev, \cny{Ukraine}}



\markboth{A.V. Ivanov, I.V. Orlovskyi}{Large deviations of
regressionparameter estimator}

\begin{abstract}
A continuous-time regression model with a jointly strictly sub-Gaussian random
noise is considered in the paper. Upper exponential bounds for
probabilities of large deviations of the least squares
estimator for the regression parameter are obtained.
\end{abstract}

\received{\sday{31} \smonth{1} \syear{2018}}
\revised{\sday{19} \smonth{4} \syear{2018}}
\accepted{\sday{22} \smonth{4} \syear{2018}}
\publishedonline{\sday{7} \smonth{5} \syear{2018}}
\end{frontmatter}

\section{Introduction}

Theory of large deviations in mathematical statistics and statistics of
stochastic processes deals with the asymptotic behaviour of tails of
distribution functions
of parametric and nonparametric statistical estimators. Concerning
parametric estimators it is necessary to refer to the monograph of
Ibragimov and Has'minskii~\cite{IbHas_SEAS} where the exponential
convergence rate of probabilities of large deviations for maximum
likelihood estimator was obtained. This result led to the appearance of
a large number of publications on the subject of large deviations of
statistical estimators.

Further we will speak about least squares estimators (l.s.e.'s) for
parameters of a nonlinear regression model. In the paper of Ivanov~\cite
{Iv_AE4DoLSEpNLRP}
a statement was proved on the power decreasing rate for
probabilities of large deviations of l.s.e.
for a scalar parameter in the nonlinear regression model with i.i.d.
observation errors having moments of finite order. Prakasa Rao~\cite
{Rao_oERCoLSEiNLRMwGE} obtained a similar result with the exponential
decreasing rate in the Gaussian nonlinear regression.

In the paper of Sieders and Dzhaparidze~\cite{SiDzha_LDR4PEaIA} a
general Theorem~2.1 on probabilities of large deviations for
$M$-estimators based on a data set of any structure was proved that
generalizes the mentioned result in \cite{IbHas_SEAS} with an
application to l.s.e. for parameters of the nonlinear regression with
pre-Gaussian and sub-Gaussian i.i.d. observation errors (Theorems 3.1
and 3.2 in \cite{SiDzha_LDR4PEaIA}). Some results in this direction are
obtained by Ivanov~\cite{Iv_AToNLR}.

The results on probabilities of large deviations
of an l.s.e. in a nonlinear regression model with correlated
observations can be found in the works of Ivanov and Leonenko~\cite
{IvLe_SAoRF}, Prakasa Rao~\cite{Rao_RoC4LSEiNLRMwDE}, Hu~\cite
{Hu_LDR4LSEiNR}, Yang and Hu~\cite{YaHu_LD4LSEiNRM}, Huang~et~al.~\cite
{Huetal_LD4LSEoNLRMBoWODE}.

Upper exponential bounds for probabilities of large deviations
of an l.s.e. for a parameter
of the nonlinear regression in discrete-time models with a jointly
strictly sub-Gaussian (j.s.s.-G.) random noise were obtained in
Ivanov~\cite{Iv_LDoRPEiMwSsGN}. In the present paper we extend some
results of~\cite{Iv_LDoRPEiMwSsGN} to continuous-time observation models.

Consider a regression model
\begin{equation}
X(t)=a(t,\,\theta)+\varepsilon(t),\quad t\ge0,
\\
\label{c_reg_m}
\end{equation}
where $a(t,\,\tau)$, $(t,\,\tau)\in\mathbb{R}_+\times\varTheta^c$, is a
continuous function, a true parameter value $\theta=(\theta_1,\ldots,\theta
_q)'$\ belongs to an open bounded convex set $\varTheta\subset\mathbb
{R}^q$ and a random noise $\varepsilon= \{\varepsilon(t),\ t\in
\mathbb{R} \}$ satisfies the following condition.

{\bf N1.} $\varepsilon$ is a mean-square and almost sure (a.s.)
continuous stochastic process defined on the probability space $
(\varOmega,\ \mathfrak{F},\ P )$, $E\varepsilon(t)=0$, $t\in\mathbb{R}$.

We shall write $\int=\int_0^T$.
\begin{defin}
Any random vector $\theta_T= (\theta_{1T},\,\ldots,\,\theta
_{qT} )'\in\varTheta^c$ having the property
\[
Q_T(\theta_T)=\underset{\tau\in\varTheta^c}
\inf{Q_T(\tau)},\: Q_T(\tau)=\int \bigl[X(t)-a(t,\,\tau)
\bigr]^2dt.
\]
is said to be the l.s.e. for an unknown parameter $\theta$, obtained by
the observations $ \{X(t),\ t\in[0,T] \}$.
\end{defin}

Under assumptions introduced above there exists at least one such
random vector $\theta_T$~\cite{Pf_oMaCoMCE}.

In the asymptotic theory of nonlinear regression in the problem of
normal approximation of the distribution
of an l.s.e., the difference $\theta_T-\theta$ is normed by diagonal
matrix~\cite{IvLe_SAoRF}
\[
d_T(\theta)=\diag \bigl(d_{iT}(\theta),\ i=\overline{1,q}
\bigr),\qquad d_{iT}^2(\theta)=\int\; \biggl(
\frac{\partial}{\partial\theta_i}a(t,\,\theta ) \biggr)^2dt.
\]
Further it is supposed that the function $a(t,\,\cdot)\in C^1(\varTheta)$
for any $t\ge0$.

The paper is organized in the following way. In Section \ref
{sec_LDiMwJSsGn} an upper exponential bound is obtained for large
deviations of $d_T(\theta) (\theta_T-\theta )$ in the
regression model \eqref{c_reg_m} with a j.s.s.-G. random noise
$\varepsilon$. In Section \ref{sec_LDitCoSJSsGN} the results of Section
\ref{sec_LDiMwJSsGn} are applied to a stationary j.s.s.-G. noise
$\varepsilon$. Section \ref{sec_2_Ex} contains examples of regression
functions $a$ and noise $\varepsilon$ satisfying the conditions of our theorems.

\section{Large deviations in models with a jointly strictly
sub-Gaussian noise}
\label{sec_LDiMwJSsGn}

\begin{defin}\label{defin_SsGRV}
A random vector $\xi= (\xi_1,\ \ldots,\ \xi_n )'\in\mathbb
{R}^n$ is called strictly sub-Gaussian (s.s.-G.) if for any $\varDelta
= (\varDelta_1,\ \ldots,\ \varDelta_n )'\in\mathbb{R}^n$
\[
E\exp \bigl\{\anglel \xi,\varDelta\angler \bigr\}\le\exp \biggl\{
\frac{1}2\anglel B\varDelta,\varDelta\angler \biggr\},
\]
where $\anglel \xi,\varDelta\angler =\sum_{i=1}^n\;\xi_i \varDelta_i$,
$B= (B(i,j) )_{i,j=1}^n$ is the covariance matrix of $\xi$,
that is $B(i,j)=E\xi_i\xi_j$, $i,j=\overline{1,n}$, $\anglel B \varDelta
,\varDelta\angler = \sum_{i,j=1}^n\;B(i,j)\varDelta_i \varDelta_j$.
\end{defin}

Note that we obtain from Definition \ref{defin_SsGRV} the definition of
an s.s.-G. random variable (r.v.) $\xi$ taking $n=1$.

\begin{defin} \label{defin_JSsGRP}
$ \{\xi(t),\ t\in\mathbb{R} \}$ is said to be jointly
strictly sub-Gaussian (j.s.s.-G.) sto\-chastic process, if for any $n\ge
1$, and any $t_1,\ \ldots,\ t_n\in\mathbb{R}$ the random vector $\xi
_n= (\xi(t_1),\ \ldots,\ \xi(t_n) )'$ is s.s.-G.
\end{defin}

These definitions and a more detailed information on sub-Gaussian
r.v.'s, vectors and stochastic processes can be found in the
monograph~\cite{BuKo_MCoRVRP} by Buldygin and Kozachenko.

Concerning the random noise $\varepsilon$ in the model \eqref{c_reg_m}
we introduce the following assumption.

{\bf N2(i)} $\varepsilon$ is a j.s.s.-G. stochastic process with the
covariance function $B(t,s)=E\varepsilon(t)\varepsilon(s)$, $t,s\in
\mathbb{R}$.

\hspace*{4mm}{\bf(ii)} For any $T>0$, $\varDelta(\cdot)\in L_2
([0,T] )$
\begin{equation}
\anglel B \varDelta,\varDelta\angler _T=\int\int\;B(t,s)\varDelta(t)
\varDelta (s)dtds\le d_0\|\varDelta\|^2_T
\label{inq_SP_By,y}
\end{equation}
for some constant $d_0>0$, $\|\varDelta\|_T= (\int\;\varDelta^2(t)dt
)^{\frac{1}2}$.

For a fixed $T$ the exact bound in \eqref{inq_SP_By,y} is
\[
\anglel B \varDelta,\varDelta\angler _T\le\|B\|_T\|
\varDelta\|^2_T,
\]
where $\|B\|_T$ is the norm of a self-adjoint positive semidefinite
operator $B$ in\break $L_2([0,T])$. Note that $\|B\|_T$ is a
nondecreasing function of $T>0$, so there exists
\[
\underset{T\to\infty}\lim\,\|B\|_T\le d_0<\infty,
\]
if \eqref{inq_SP_By,y} is fulfilled.

\begin{exm} Assume the covariance function $B(t,s)$ is such that
\[
\text{1)}\quad b_1^2=\int\limits
_0^\infty\int\limits
_0^\infty
\; B^2(t,s)dtds<\infty\qquad \text{or\quad 2)}\quad b_2=\underset{t\in
\mathbb{R}_+}\sup\; \int\limits
_0^\infty\;\bigl|B(t,s)\bigr|ds<\infty.
\]
Then using condition 1) and the Fubini theorem we get
\[
\begin{aligned} \anglel B \varDelta,\varDelta\angler _T&=
\int \biggl(\int\;B(t,s)\varDelta(s)ds \biggr)\varDelta(t) dt
\\
&\le \biggl(\int \biggl(\int\;B(t,s)\varDelta(s)ds \biggr)^2dt
\biggr)^{\frac{1}{2}} \|\varDelta\|_T\le b_1\|
\varDelta\|_T^2, \end{aligned} %
\]
and we can take $d_0=b_1$. On the other hand,
\[
\anglel B \varDelta,\varDelta\angler _T\le\int\int\;\bigl|B(t,s)\bigr|
\varDelta^2(t)dtds \le b_2\|\varDelta\|_T^2,
\]
and we can take $d_0=b_2$.
\end{exm}

Let $\varDelta(t),\ t\in\mathbb{R}_+$, be a continuous function. Then
condition {\bf N1} implies the existence of the integral
\begin{equation}
I(T)=\int\;\varDelta(t)\varepsilon(t)dt, \label{def_I_T}
\end{equation}
determined for almost all paths of the process $\varepsilon(t),\ t\in
[0,T]$, as the Riemann integral. Consider partitions $r(n)$
\[
0=t_0^{(n)}<t_1^{(n)}<
\cdots<t_n^{(n)}=T
\]
of the interval $[0,T]$ such that $\max_{1\le k\le n}
(t_k^{(n)}-t_{k-1}^{(n)} )\to0$, as $n\to\infty$, and the
corresponding integral sums
\[
I_n(T)=\sum\limits
_{k=1}^n
\,u_k^{(n)}\varepsilon \bigl(t_k^{(n)}
\bigr),\qquad u_k^{(n)}=\varDelta \bigl(t_k^{(n)}
\bigr) \bigl(t_k^{(n)}-t_{k-1}^{(n)}
\bigr),\qquad k=\overline{1,n}.
\]
Then
\begin{equation}
I_n(T)\to I(T)\ \text{a.s.},\quad \text{as}\ n\to\infty. \label{I_nT_2_I_T_as}
\end{equation}

It is obvious also that
\begin{equation}
EI_n^2(T)\to EI^2(T)=\anglel B \varDelta,
\varDelta\angler _T,\quad \text{as}\ n\to \infty. \label{EI_nT^2_2_EI_T^2}
\end{equation}

\begin{lemma}\label{lema_Int_Dlta_Eps_SsG}
Under conditions {\bf N1}, {\bf N2} integral \eqref{def_I_T} is an s.s.-G.
r.v., for any $T>0$.
\end{lemma}

\begin{proof}
From Definition \ref{defin_JSsGRP} it follows that the process
$\varepsilon(t)$, $t\in[0,T]$, is j.s.s.-G., if for any $n\ge1$, and
$t_1,\ \ldots,\ t_n\in[0,T]$, $u_1,\ \ldots,\ u_n\in\mathbb{R}$,
$\lambda\in\mathbb{R}$,
\[
E\exp \Biggl\{\lambda\sum\limits
_{k=1}^n
\,u_k\varepsilon(t_k) \Biggr\}\le \exp \Biggl\{
\frac{1}2\lambda^2\sum\limits
_{i,j=1}^n
\,B(t_i,\, t_j)u_iu_j \Biggr
\}.
\]
Taking $t_k=t_k^{(n)}$, $u_k=u_k^{(n)}$, $k=\overline{1,n}$, we obtain
\[
E\exp \bigl\{\lambda I_n(T) \bigr\}\le\exp \biggl\{\frac{1}2
\lambda^2E I_n^2(T) \biggr\}.
\]
Due to \eqref{I_nT_2_I_T_as} and \eqref{EI_nT^2_2_EI_T^2} by the Fatou
lemma (see, for example,~\cite{GiSko_I2tToRP})
%
\begin{align*}
E\exp\bigl\{\lambda I(T)\bigr\}&=E\underset{n\to\infty}\lim\,\exp\bigl\{\lambda
I_n(T)\bigr\} \le\underset{n\to\infty}\liminf\,E\exp\bigl\{\lambda
I_n(T)\bigr\}
\\
&\le\underset{n\to\infty}\lim\,\exp \biggl\{\frac{1}2
\lambda^2E I_n^2(T) \biggr\}=\exp \biggl\{
\frac{1}2\lambda^2E I^2(T) \biggr\}.\qedhere
\end{align*}
%
\end{proof}

The following statement on the exponential bound
for distribution tails of integrals \eqref{def_I_T}
plays an important role in subsequent proofs.

\begin{lemma}\label{lema4P_I_T}
Under conditions {\bf N1} and {\bf N2} for any $T>0$, $x>0$,
\begin{equation}
\begin{aligned}
P \bigl\{I(T)\ge x \bigr\}&\le G_T(x),\qquad P
\bigl\{I(T)\le-x \bigr\}\le G_T(x),
\\
P \bigl\{ \bigl\llvert I(T) \bigr\rrvert\ge x \bigr\}&\le2G_T(x),
\end{aligned} %
\label{inq_P_I_T}
\end{equation}
where
\[
G_T(x)=\exp \biggl\{-\frac{x^2}{2d_0\|\varDelta\|_T^2} \biggr\}.
\]
\end{lemma}
\begin{proof}
The proof is obvious (see, for example,~\cite{BuKo_MCoRVRP}). For any
$x>0$, $\lambda>0$ by the Chebyshev--Markov inequality, \eqref
{inq_SP_By,y}, and Lemma \ref{lema_Int_Dlta_Eps_SsG}
\begin{equation}
P \bigl\{I(T)\ge x \bigr\}\le\exp\{-\lambda x\}\exp \biggl\{\frac{1}2
\lambda ^2\anglel B \varDelta,\varDelta\angler _T \biggr\}
\le\exp \biggl\{\frac{1}2\lambda ^2d_0\|\varDelta
\|_T^2-\lambda x \biggr\}. \label{inq_P_I_T_d_lmba}
\end{equation}
Minimization of the right-hand side of \eqref{inq_P_I_T_d_lmba} in
$\lambda$ gives the first inequality in \eqref{inq_P_I_T}. The proof of
the second inequality is similar. The third inequality follows from the
previous ones.
\end{proof}

We need some notation to formulate conditions on the regression
function $a(t,\,\theta)$ using the approach of the paper~\cite
{SiDzha_LDR4PEaIA} (see also~\cite{Iv_AToNLR,IvLe_SAoRF}). Write
\[
U_T(\theta)=d_T(\theta) \bigl(\varTheta^c-
\theta \bigr),\qquad \varGamma_{T,\theta
,R}=U_T(\theta)\cap\{u\,:\,R\le
\|u\|\le R+1\},
\]
$u=(u_1,\ \ldots,\ u_q)'\in\mathbb{R}^q$. Denote by $G$ the family of
all functions $g=g_T(R)$, $T>0$, $R>0$, having the following properties:

1) for fixed $T$, $g_T(R)\uparrow\infty$, as $R\to\infty$;

2) for any $r>0$,
\[
R^r\exp \bigl\{-g_T(R) \bigr\}\to0,\quad \text{as}\ R,\,T\to
\infty.
\]
Let $\gamma(R)$ be polynomials of $R$ (possibly different) with
coefficients that do not depend on values $T,\ \theta,\ u,\ v$. Set also
\[
\begin{aligned} \varDelta(t,\,u)&=a\bigl(t,\,\theta+d_T^{-1}(
\theta)u\bigr)-a(t,\,\theta),\quad t\in [0,T],
\\
\varPhi_T(u,v)&= \int \bigl(\varDelta(t,\,u)-\varDelta(t,\,v)
\bigr)^2dt,\quad u,v\in U_T(\theta). \end{aligned} %
\]

Assume the existence of a function $g\in G$, constants $\delta\in
(0,\frac{1}2 )$, $\varkappa>0$, $\rho\in(0,1]$ and polynomials
$\gamma(R)$ such that for sufficiently large $T,\ R$ (we will write
$T>T_0$, $R>R_0$) the following conditions are fulfilled.

{\bf R1(i)} For any $u,v\in\varGamma_{T,\theta,R}$ such that $\|u-v\|\le
\varkappa$
\begin{equation}
\varPhi_T(u,v)\le\gamma(R)\|u-v\|^{2\rho}; \label{inq_Phi_uv<=u-v_p}
\end{equation}

\hspace*{6mm}{\bf(ii)} for any $u\in\varGamma_{T,\theta,R}$ $\varPhi
_T(u,0)\le\gamma(R)$.

{\bf R2.} For any $u\in\varGamma_{T,\theta,R}$
\begin{equation}
\varPhi_T(u,0)\ge2d_0\delta^{-2}g_T(R).
\label{inq_Phi_u0_>=g_T}
\end{equation}

\begin{thm} \label{thm_PLD_sGN}
If conditions {\bf N1}, {\bf N2}, {\bf R1} and {\bf R2} are fulfilled, then there
exist constants $B_0$, $b_0>0$ such that for $T>T_0$, $R>R_0$
\begin{equation}
P \bigl\{ \bigl\llVert d_T(\theta) (\theta_T- \theta )
\bigr\rrVert \ge R \bigr\} \le B_0\exp \bigl\{-b_0g_T(R)
\bigr\}, \label{inq_gPLD_sGN}
\end{equation}
where for any $\beta>0$ the constant $B_0$ can be chosen so that
\begin{equation}
b_0\ge\frac{\rho}{\rho+q}-\beta. \label{inq4a_0dep_on_A_0}
\end{equation}
\end{thm}

\begin{proof}
Set
\[
I(T,\,u)=\int\;\varDelta(t,\,u)\varepsilon(t)dt,\qquad \zeta_T(u)=I(T,
\,u)-\frac{1}2\varPhi_T(u,0).
\]
To prove the theorem it is sufficient to check the fulfilment of
assumptions (M1) and (M2) of the Theorem 2.1 in \cite
{SiDzha_LDR4PEaIA}, reformulated in the manner similar to that used in
the proof of Theorem 3.1, ibid.:

{\bf(M1)} for any $m>0$ and $u,v\in\varGamma_{T,\theta,R}$ such that $\|
u-v\|\le\varkappa$,
\begin{equation}
E \bigl\llvert \zeta_T(u)-\zeta_T(v) \bigr\rrvert
^m\le\gamma(R)\|u-v\|^{\rho m}; \label{asum_M1_4LSE}
\end{equation}

{\bf(M2)}
\begin{equation}
P \biggl\{\zeta_T(u)-\zeta_T(0)\ge- \biggl(
\frac{1}2-\delta \biggr)\varPhi _T(u,0) \biggr\}\le\exp \bigl
\{-g_T(R) \bigr\}. \label{asum_M2_4LSE}
\end{equation}

From the first inequality in \eqref{inq_P_I_T} of Lemma \ref
{lema4P_I_T} for $\varDelta(t)=\varDelta(t,\,u)$, $x=\delta\varPhi_T(u,0)$,
condition {\bf R2}, taking into account that $\zeta_T(0)=0$ in our
particular case, we obtain
\[
\begin{aligned}
P \biggl\{\zeta_T(u)-
\zeta_T(0)\ge- \biggl(\frac{1}2-\delta \biggr)\varPhi
_T(u,0) \biggr\}
&=P \bigl\{I(T,\,u)\ge\delta\varPhi_T(u,0)
\bigr\}
\\
&\le\exp \bigl\{-\delta^2 (2d_0 )^{-1}
\varPhi_T(u,0) \bigr\}\\
&\le\exp \bigl\{-g_T(R) \bigr\},
\end{aligned} %
\]
i.e. \eqref{asum_M2_4LSE} is true.

On the other hand,
%
\begin{align}
&E \bigl\llvert \zeta_T(u)-
\zeta_T(v) \bigr\rrvert ^m\notag\\
&\quad\le\max\bigl(1,2^{m-1}\bigr)\cdot \bigl(E \bigl\llvert I(T,\,u)-I(T,\,v)
\bigr\rrvert ^m+2^{-m} \bigl\llvert \varPhi_T(u,0)-
\varPhi_T(v,0) \bigr\rrvert ^m \bigr), 
\label{1inq_E_dif_zeta^m}\\
&\bigl\llvert \varPhi_T(u,0)-
\varPhi_T(v,0) \bigr\rrvert \notag\\
&\quad\le\int\;\bigl|\varDelta(t,\,u)-\varDelta (t,
\,v)\bigr|\cdot\bigl|\varDelta(t,\,u)+\varDelta(t,\,v)\bigr|dt\notag\\
&\quad\le\varPhi_T^{\frac{1}2}(u,v)\cdot \bigl(\varPhi_T^{\frac{1}2}(u,0)+
\varPhi_T^{\frac{1}2}(v,0) \bigr)\notag\\
&\quad\le2\bigl(\gamma(R)
\bigr)^{\frac{1}2}\|u-v\|^\rho\bigl(\gamma(R)\bigr)^{\frac{1}2}\notag\\
&\quad\le\bigl(\gamma(R)+\gamma(R)\bigr)\|u-v\|^\rho=\gamma(R)\|u-v
\|^\rho\notag
\end{align} 
according to {\bf R1} (polynomials $\gamma(R)$ are different in the
last two lines). Thus we obtain the bound
\begin{equation}
\bigl\llvert \varPhi_T(u,0)-\varPhi_T(v,0) \bigr\rrvert
^m\le\gamma(R)\|u-v\|^{\rho m}. \label{inq_dif_Phi^m}
\end{equation}

By the formula for the moments of nonnegative r.v. (see, for example,
\cite{Fe_I2PTaIA} and compare with \cite{Hu_LDR4LSEiNR}) and the third
inequality of Lemma \ref{lema4P_I_T} being applied to $\varDelta(t)=\varDelta
(t,\,u)-\varDelta(t,\,v)$, $t\in[0,T]$, it holds
\begin{equation}
\begin{aligned}
E \bigl\llvert I(T;u,v) \bigr\rrvert^m &=m\int\limits
_0^\infty\;x^{m-1}P \bigl\{ \bigl\llvert
I(T;u,v) \bigr\rrvert \ge x \bigr\}dx\\
&\le2m\int\limits
_0^\infty\;x^{m-1}\exp \biggl\{-
\frac{x^2}{2d_0\varPhi
_T(u,v)} \biggr\}dx \\
&=\sqrt{2\pi}m d_0^{\frac{m}2}
\varPhi_T^{\frac{m}2}(u,v)E|Z|^{m-1}, \end{aligned}
\label{inq_EI_T-uv^m}
\end{equation}
where $I(T;u,v)=I(T,\,u)-I(T,\,v)$, $Z$ is a standard Gaussian r.v.,
\begin{equation}
E|Z|^{m-1}=\pi^{-\frac{1}2}2^{\frac{m-1}2}\varGamma \biggl(
\frac{m}2 \biggr) ,\quad m>0. \label{E_StGaus^m-1}
\end{equation}
Relations \eqref{inq_EI_T-uv^m}, \eqref{E_StGaus^m-1}, and \eqref
{inq_Phi_uv<=u-v_p} lead to the bound
\begin{equation}
E \bigl\llvert I(T;u,v) \bigr\rrvert ^m\le2^{\frac{m}2}m\varGamma
\biggl(\frac{m}2 \biggr)d_0^{\frac{m}2}
\varPhi_T^{\frac{m}2}(u,v)\le\gamma(R)\|u-v\|^{\rho m}.
\label{f_inq_EI_T-uv^m}
\end{equation}
From \eqref{1inq_E_dif_zeta^m}, \eqref{inq_dif_Phi^m}, and \eqref
{f_inq_EI_T-uv^m} it follows \eqref{asum_M1_4LSE}.
\end{proof}

Suppose there exist a diagonal matrix $s_T=\diag (s_{iT},\
i=\overline{1,q} )$ with elements that do not depend on $\tau\in
\varTheta$, and constants $0<\underline{c}_i<\overline{c}_i<\infty$,
$i=\overline{1,q}$, such that uniformly in $\tau\in\varTheta$ for $T>T_0$
\begin{equation}
\underline{c}_i<s_{iT}^{-1}d_{iT}(
\tau)<\overline{c}_i,\quad i=\overline{1,q}. \label{inq4s_iT&d_iT}
\end{equation}
Then instead of the matrix $d_T(\theta)$ (at least in the framework of
the topic of this paper) it is possible to consider, without loss of
generality, the normalizing matrix $s_T$.

The next condition is more restrictive than {\bf R1} and {\bf R2},
however it is simpler due to requirement \eqref{inq4s_iT&d_iT}.

{\bf R3.} There exist numbers $0<c_0<c_1<\infty$ such that for any
$u,v\in U_T(\theta)=s_T(\varTheta^c-\theta)$ and $T>T_0$,
\begin{equation}
c_0\|u-v\|^2\le\varPhi_T(u,v)\le
c_1\|u-v\|^2. \label{inq4int_dif_H^2}
\end{equation}

It goes without saying that in the expression for $\varPhi_T(u,v)$ in
\eqref{inq4int_dif_H^2} we use the matrix $s_T^{-1}$ instead of
$d_T^{-1}(\theta)$.

A condition of the type \eqref{inq4int_dif_H^2} has been introduced in
\cite{Iv_AE4DoLSEpNLRP} and used in \cite
{Rao_oERCoLSEiNLRMwGE,SiDzha_LDR4PEaIA,Hu_LDR4LSEiNR} and other works.
The next theorem generalizes Theorem 3.2 from \cite{SiDzha_LDR4PEaIA}.

\begin{thm}\label{thm_sPLD_sGN}
Under conditions {\bf N1}, {\bf N2} and {\bf R3} there exist constants $B$,
$b>0$ such that for $T>T_0$, $R>R_0$
\begin{equation}
P \bigl\{ \bigl\llVert s_T (\theta_T- \theta ) \bigr
\rrVert \ge R \bigr\} \le B\exp \bigl\{-b R^2 \bigr\}, \label{inq_PLD_sGN}
\end{equation}
moreover for any $\beta>0$ the constant $B$ can be chosen so that
\begin{equation}
b\ge\frac{c_0} {8d_0(1+q)}-\beta. \label{inq4a_dep_on_A_0}
\end{equation}
\end{thm}

\begin{proof}
We will show that {\bf R3} implies
conditions {\bf R1} and {\bf R2}. Inequality \eqref{inq_Phi_uv<=u-v_p}
of the condition {\bf R1(i)} follows from the right-hand side of
inequality \eqref{inq4int_dif_H^2}, if we take $\rho=1$, $\gamma
(R)=c_1$. Inequality of the condition {\bf R1(ii)} follows as well from
the right-hand side of \eqref{inq4int_dif_H^2}, if we take $v=0$,
$\gamma(R)=c_1(R+1)^2$.

To check the fulfilment of condition {\bf R2} we should rewrite the
left-hand side of \eqref{inq4int_dif_H^2} for $v=0$:
\[
\varPhi_T(u,0)\ge c_0\|u\|^2
\ge2d_0\delta^{-2} \biggl(\frac{1}2\delta
^2d_0^{-1}c_0R^2
\biggr),
\]
i.e., in the inequality \eqref{inq_Phi_u0_>=g_T} one can take
$g_T(R)=\frac{1}2\delta^2d_0^{-1}c_0R^2$. In this case for the exponent
in \eqref{inq_gPLD_sGN} we have
\[
-b_0g_T(R)=- \biggl(\frac{1}2
\delta^2b_0d_0^{-1}c_0
\biggr)R^2.
\]
Now, since $\rho=1$ in \eqref{inq4a_0dep_on_A_0}, for any $\beta>0$ in
\eqref{inq_PLD_sGN} we can take
\[
b=b_\delta=\frac{1}2\delta^2b_0d_0^{-1}c_0
\ge\frac{\delta^2c_0}{
2d_0(1+q)}-\beta.
\]
By {\bf R3} and {\bf N1}, {\bf N2}, inequality \eqref{inq_Phi_u0_>=g_T} with
$g_T(R)=\frac{1}2\delta^2d_0^{-1}c_0R^2$ holds for any $\delta\in
(0,\frac{1}2 )$. We get inequality \eqref{inq4a_dep_on_A_0} as
$\delta\to\frac{1}2$.
\end{proof}

\section{Large deviations in the case of a stationary jointly strictly
sub-Gaussian noise}
\label{sec_LDitCoSJSsGN}

We impose an additional restriction on the noise process $\varepsilon$.

{\bf N3.} The stochastic process $\varepsilon$ is stationary with the
covariance function $B(t)=E\varepsilon(0)\varepsilon(t),\ t\in\mathbb
{R}$, and the bounded spectral density $f(\lambda),\ \lambda\in\mathbb{R}$:
\[
f_0=\underset{\lambda\in\mathbb{R}}\sup\;f(\lambda)<\infty.
\]

Under assumption {\bf N3} the following corollaries of the theorems
proved in Section \ref{sec_LDiMwJSsGn} are true.

\begin{cor} \label{cor_PLD_gSt_sGN}
If conditions {\bf N1}, {\bf N2(i)}, {\bf N3}, {\bf R1} and {\bf R2} are fulfilled, then
the statement of Theorem \ref{thm_PLD_sGN} is true with $d_0=2\pi f_0$.
\end{cor}

\begin{proof}
We just need to show that condition {\bf N2(ii)} is fulfilled. Indeed,
by the\break Plancherel identity,
\[
\anglel B \varDelta,\varDelta\angler _T=\int\limits
_{-\infty}^\infty
\;f(\lambda ) \biggl\llvert \int\;e^{i\lambda t}\varDelta(t) dt \biggr\rrvert
^2d\lambda\le2\pi f_0\| \varDelta\|_T^2.
\qedhere
\]
\end{proof}

\begin{cor} \label{cor_sPLD_gSt_sGN}
Under conditions {\bf N1}, {\bf N2(i)}, {\bf N3} and {\bf R3} the statement of
Theorem \ref{thm_sPLD_sGN} is true with inequality \eqref
{inq4a_dep_on_A_0} rewritten in the form
\begin{equation}
b\ge\frac{c_0} {16\pi f_0(1+q)}-\beta. \label{inq4a_dep_on_A_0_St}
\end{equation}
\end{cor}

Our next assumption is a particularization of the requirements {\bf N2}
and {\bf N3}.

{\bf N4(i).} The random noise $\varepsilon$ is of the form
\begin{equation}
\varepsilon(t)=\int\limits
_{-\infty}^t\;\psi(t-s)d\xi(s)=\int\limits
_0^\infty
\;\psi(s)d\xi(t-s), \label{int_repr_epsln}
\end{equation}
where $\xi= \{\xi(t),\ t\in\mathbb{R} \}$ is a mean-square
continuous j.s.s.-G. stochastic process with orthogonal increments,
$E\xi(t)=0$,
\[
E\bigl|\xi(t+s)-\xi(t)\bigr|^2=s,\quad t\in\mathbb{R},\ s>0;
\]
$\psi(t)=0$ as $t<0$ and
\[
\int\limits
_0^\infty\;\psi^2(t)dt<\infty.
\]

The stochastic integral in \eqref{int_repr_epsln} is understood as a
mean-square Stieltjes integral \cite{GiSko_I2tToRP}. The process $\xi$
is an integrated white noise, $\varepsilon$ can be considered as a
stationary process at the output of a physically realizable filter with
the covariance function (see ibid.)
\[
B(t)=\int\limits
_0^\infty\;\psi(t+u)\psi(u)du
\]
and the spectral density
\[
f(\lambda)= \bigl\llvert h(i\lambda) \bigr\rrvert ^2,\qquad h(i
\lambda)=(2\pi)^{-\frac{1}2}\int\limits
_0^\infty\;\psi(t)e^{-i\lambda t}dt.
\]

{\bf N4(ii).} $f_0=\sup_{\lambda\in\mathbb{R}} \llvert h(i\lambda
) \rrvert ^2<\infty$.

Obviously {\bf N4(ii)} holds, if $\int_0^\infty\, \llvert \psi
(t) \rrvert dt<\infty$.

\begin{lemma}\label{lema_Int2jssG}
If condition {\bf N4(i)} holds, then $\varepsilon$ in \eqref
{int_repr_epsln} is a j.s.s.-G. process.
\end{lemma}

\begin{proof}
Let $n\ge1$ be a fixed number and $t_1,\ \ldots,\ t_n$, $\varDelta_1,\
\ldots,\ \varDelta_n$ be arbitrary real numbers. It is necessary to prove that
\begin{equation}
E\exp \Biggl\{\sum\limits
_{k=1}^n\;\varDelta_k
\varepsilon(t_k) \Biggr\}\le \exp \Biggl\{\frac{1}2\sum
\limits
_{i,j=1}^n\;B(t_i-t_j)
\varDelta_i \varDelta _j \Biggr\}. \label{cndn4sGP4eps}
\end{equation}
Formula \eqref{int_repr_epsln} can be rewritten in the form
\[
\varepsilon(t)=\int\limits
_{-\infty}^\infty\;\psi(t-s)d\xi(s),\quad \psi (t)=0,\ \text{as}
\ t<0.
\]
Denote $\psi_k(s)=\psi (t_k-s )$, $k=\overline{1,n}$. Then
\[
\varepsilon (t_k )=\int\limits
_{-\infty}^\infty\;\psi_k(s)d
\xi (s),\quad k=\overline{1,n}.
\]
Let a sequence of simple functions
\[
\psi_k^{(m)}(s)=\sum\limits
_{l=1}^{r(k,m)}
\;c_{kl}^{(m)}\chi_{\pi
_{kl}^{(m)}}(s),\quad m\ge1,
\]
where $\pi_{kl}^{(m)}=[\alpha_{kl}^{(m)},\,\beta_{kl}^{(m)}
)$, $k=\overline{1,n}$, $l=\overline{1,r(k,m)}$, $\chi_A(s)$ is
indicator of a set $A$, approximate the function $\psi_k(s)$ in
$L_2(\mathbb{R})$:
\[
\int\limits
_{-\infty}^\infty\; \bigl\llvert \psi_k(s)-
\psi_k^{(m)}(s) \bigr\rrvert ^2ds\to0,\quad
\textnormal{as}\ m\to\infty.
\]
Then the sequences of random variables
\[
\varepsilon_k^{(m)}=\sum\limits
_{l=1}^{r(k,m)}
\; c_{kl}^{(m)} \bigl(\xi \bigl(\beta_{kl}^{(m)}
\bigr)- \xi \bigl(\alpha_{kl}^{(m)} \bigr) \bigr) = \int
\limits_{-\infty}^\infty
\;\psi_k^{(m)}(s)d\xi(s)
\]
mean-square converge to $\varepsilon(t_k)$ in $L_2(\varOmega)$:
\begin{equation}
E \bigl\llvert \varepsilon(t_k)-\varepsilon_k^{(m)}
\bigr\rrvert ^2\to0,\qquad k=\overline {1,n},\quad \text{as}\ m\to\infty.
\label{cnv_eps_m_2_eps_L2}
\end{equation}

For any fixed $m$, the random vector with coordinates $\varepsilon
_k^{(m)}$, $k=\overline{1,n}$ is s.s.-G. Indeed,
\[
\begin{aligned} \sum\limits
_{k=1}^m\,
\varDelta_k\varepsilon_k^{(m)}&=\sum
\limits
_{k=1}^m\, \varDelta_k\sum
\limits
_{l=1}^{r(k,m)}\; c_{kl}^{(m)} \bigl(
\xi \bigl(\beta _{kl}^{(m)} \bigr)- \xi \bigl(
\alpha_{kl}^{(m)} \bigr) \bigr)
\\
&=\sum\limits
_{k'=1}^{n'(m)}\,u_{k'}^{(m)}
\xi \bigl(\eta _{k'}^{(m)} \bigr), \end{aligned} %
\]
where $u_{k'}^{(m)}$ are real numbers and $\eta_{k'}^{(m)}$ are
different real numbers. By condition {\bf N4(i)} the random vector with
coordinates $\xi (\eta_{k'}^{(m)} )$, $k'=\overline{1,n'(m)}$,
is s.s.-G., and therefore,
\begin{equation}
\begin{aligned} E\exp \Biggl\{\sum\limits
_{k=1}^m
\,\varDelta_k\varepsilon_k^{(m)} \Biggr\}&= E
\exp \Biggl\{\sum\limits
_{k'=1}^{n'(m)}\,u_{k'}^{(m)}
\xi \bigl(\eta _{k'}^{(m)} \bigr) \Biggr\}
\\
&\le\exp \Biggl\{\frac{1}2E \Biggl(\sum\limits
_{k'=1}^{n'(m)}
\, u_{k'}^{(m)}\xi \bigl(\eta_{k'}^{(m)}
\bigr) \Biggr)^2 \Biggr\} \\
&=\exp \Biggl\{\frac{1}2 \Biggl(\sum
\limits
_{k=1}^m\,\varDelta_k
\varepsilon_k^{(m)} \Biggr)^2 \Biggr\}.
\end{aligned} %
\label{eps_k^m_is_SsG}
\end{equation}
From \eqref{cnv_eps_m_2_eps_L2} it follows that $\varepsilon
_k^{(m)}\overset{P}\to\varepsilon_{t_k}$, $k=\overline{1,n}$, as $m\to
\infty$, and thus there exists some subsequence of indexes $m^\prime\to
\infty$,
independent of $k$, such that $\varepsilon_k^{(m')}\to\varepsilon
_{t_k}$ a.s., $k=\overline{1,n}$, as $m^\prime\to\infty$.

Finally, by the Fatou lemma and \eqref{eps_k^m_is_SsG}
\[
\begin{aligned} E\exp \Biggl\{\sum\limits
_{k=1}^n
\;\varDelta_k\varepsilon(t_k) \Biggr\} &=E
\underset{m^\prime\to\infty}\lim\;\exp \Biggl\{\sum
\limits_{k=1}^n
\; \varDelta_k\varepsilon_k^{(m')} \Biggr\}
\\
&\le\underset{m^\prime\to\infty}\liminf\;E\exp \Biggl\{\sum
\limits
_{k=1}^n\;\varDelta_k\varepsilon_k^{(m')}
\Biggr\}\\
&\le \underset{m^\prime\to\infty}\lim\;\exp \Biggl\{
\frac{1}2E \Biggl(\sum\limits
_{k=1}^n\;
\varDelta_k\varepsilon_k^{(m')}
\Biggr)^2 \Biggr\}
\\
&=\exp \Biggl\{\frac{1}2E \Biggl(\sum\limits
_{k=1}^n
\;\varDelta_k\varepsilon (t_k) \Biggr)^2
\Biggr\}\\
&=\exp \Biggl\{\frac{1}2 \sum\limits
_{i,j=1}^n
\; B(t_i-t_j)\varDelta_i
\varDelta_j \Biggr\}. \end{aligned} %
\]
So, we have obtained \eqref{cndn4sGP4eps}.
\end{proof}

\begin{cor} \label{cor_PLD_St_sGN}
If conditions {\bf N1}, {\bf N4(i)}, {\bf N4(ii)}, {\bf R1} and {\bf R2} are fulfilled,
then the conclusion of Theorem \ref{thm_PLD_sGN} is true with $d_0=2\pi f_0$.
\end{cor}

\begin{cor} \label{cor_sPLD_St_sGN}
If conditions {\bf N1}, {\bf N4(i)}, {\bf N4(ii)} and {\bf R3} are fulfilled, then
the conclusion of Theorem \ref{thm_sPLD_sGN} is true with a constant
$b$ satisfying \eqref{inq4a_dep_on_A_0_St}.
\end{cor}

Assume
\begin{equation}
\underset{T\to\infty} {\liminf}\;T^{-\frac{1}2}s_{iT}>0,\quad i=
\overline{1,q}. \label{liminf_T^-1/2s_iT}
\end{equation}

\begin{cor}\label{cor_LSE_Spd_of_Cnv}
Under conditions of Theorem \ref{thm_sPLD_sGN} or Corollaries \ref
{cor_sPLD_gSt_sGN}, \ref{cor_sPLD_St_sGN}, and \eqref
{liminf_T^-1/2s_iT} for any $\rho>0$, $\nu\in [0,\frac{1}2 )$,
and $T>T_0$
\begin{equation}
P \bigl\{ \bigl\llVert T^{-\frac{1}2}s_T (\theta_T-
\theta ) \bigr\rrVert \ge \rho T^{-\nu} \bigr\} \le B\exp \bigl\{-b\rho
T^{1-2\nu} \bigr\}. \label{inq_PLD4cons_rho_dlta}
\end{equation}
\end{cor}

\begin{proof}
To show \eqref{inq_PLD4cons_rho_dlta} it is sufficient to take $R=\rho
T^{\frac{1}2-\nu}$ in \eqref{inq_PLD_sGN}.
\end{proof}

For $\nu=0$ we arrive at a quite strong result on
the weak consistency of l.s.e. Similarly, in the conditions of
Corollary \ref{cor_LSE_Spd_of_Cnv} the following result on
probabilities of moderate deviations for l.s.e. holds: for any $h>0$, $T>T_0$
\[
P \bigl\{ \bigl\llVert s_T (\theta_T- \theta ) \bigr
\rrVert \ge h \ln^{\frac{1}2}T \bigr\} \le B T^{-bh^2}.
\]

Obviously, Gaussian stochastic processes $\varepsilon$ are j.s.s.-G.
ones, and all the previous results are valid for them.

\section{Two examples}
\label{sec_2_Ex}

In this section, we consider an example of a regression
function satisfying the condition {\bf R3} and an example of the
j.s.s.-G. process $\xi$ from expression \eqref{int_repr_epsln} in
condition {\bf N4(i)}.

\begin{exm} Suppose
\[
a(t,\,\tau)=\exp \bigl\{\bigl\anglel \tau,\,y(t)\bigr\angler \bigr\},
\]
where $\anglel \tau,\,y(t)\angler =\sum_{i=1}^q\,\tau_iy_i(t)$,
regressors $y(t)= (y_1(t),\,\ldots,\,y_q(t) )'$, $t\ge0$, take
values in a compact set $Y\subset\mathbb{R}^q$.

Let us assume
\begin{equation}
J_T= \biggl(T^{-1}\int\limits
_0^Ty_i(t)y_j(t)dt
\biggr)_{i,j=1}^q\ \to\ J= (J_{ij}
)_{i,j=1}^q,\quad \text{as}\ T\to\infty, \label{cnv_J_T_2_J}
\end{equation}
$J$ is a positive definite matrix. In this case the regression function
$a(t,\,\tau)$ satisfies condition {\bf R3}. Indeed, let
\[
H=\underset{y\in Y,\ \tau\in\varTheta^c}\max\,\exp \bigl\{\anglel y,
\,\tau \angler \bigr\},\qquad L=\underset{y\in Y,\ \tau\in\varTheta^c}\min
\,\exp \bigl\{ \anglel y,\,\tau\angler \bigr\}
\]
Then for any $\delta>0$ and $T>T_0$
\[
L^2 (J_{ii}-\delta )< T^{-1}
d_{iT}^2(\theta)< H^2 (J_{ii}+\delta
),\quad i=\overline{1,q},
\]
and according to \eqref{inq4s_iT&d_iT} we can take $s_T=T^{\frac{1}2}\mathbb{I}_q$
with the identity matrix $\mathbb{I}_q$ of order $q$.

For a fixed $t$
\[
\begin{aligned}
&\exp \bigl\{\bigl\anglel y(t),\theta+T^{-\frac{1}2}u
\bigr\angler \bigr\}-\exp \bigl\{ \bigl\anglel y(t),\theta+T^{-\frac{1}2}v\bigr
\angler \bigr\}
\\
&\quad =T^{-\frac{1}2}\sum\limits
_{i=1}^q
\,y_i(t)\exp \bigl\{\bigl\anglel y(t),\theta +T^{-\frac{1}2}
\bigl(u+\eta(v-u) \bigr)\bigr\angler \bigr\} (u_i-v_i ),
\ \eta\in(0,1), \end{aligned} %
\]
and therefore for any $\delta>0$ and $T>T_0$
\[
\varPhi_T(u,v)\le H^2 \biggl(T^{-1}\int\;
\bigl\|y(t)\bigr\|^2dt \biggr)\|u-v\|^2< H^2 (\Tr J+
\delta )\|u-v\|^2,
\]
So we obtain the right-hand side of \eqref{inq4int_dif_H^2} with the
constant $c_1>H^2\Tr J$.

On the other hand, for a fixed $t$
\[
\begin{aligned}
& \bigl(\varDelta(t,u)-\varDelta(t,v) \bigr)^2
\\
&\quad = \bigl(\exp \bigl\{\bigl\anglel y(t),\theta+T^{-\frac{1}2}u\bigr
\angler \bigr\} -\exp \bigl\{\bigl\anglel y(t),\theta+T^{-\frac{1}2}v\bigr
\angler \bigr\} \bigr)^2
\\
&\quad =\exp \bigl\{2\bigl\anglel y(t),\theta+T^{-\frac{1}2}v\bigr\angler
\bigr\} \bigl(\exp \bigl\{\bigl\anglel y(t),T^{-\frac{1}2}(u-v)\bigr\angler
\bigr\}-1 \bigr)^2. \end{aligned} %
\]
Since $ (e^x-1 )^2\ge x^2$, $x\ge0$, and $ (e^x-1
)^2\ge e^{2x}x^2$, $x<0$, it holds
\[
\bigl(\exp \bigl\{\bigl\anglel y(t),T^{-\frac{1}2}(u-v)\bigr\angler \bigr\}-1
\bigr)^2\ge L_tT^{-1}\bigl\anglel y(t),u-v\bigr
\angler ^2,
\]
with $L_t=\min \{1,\,\exp \{2\anglel y(t),T^{-\frac{1}2}(u-v)\angler  \} \}$, and
\[
\begin{aligned}
& \bigl(\varDelta(t,u)-\varDelta(t,v)\bigr)^2
\\
&\quad \ge\min \bigl\{\exp \bigl\{2\bigl\anglel y(t),\theta+T^{-\frac{1}2}v
\bigr\angler \bigr\},\,\exp \bigl\{2\bigl\anglel y(t),\theta+T^{-\frac{1}2}u
\bigr\angler \bigr\} \bigr\} T^{-1}\bigl\anglel y(t),u-v\bigr\angler
^2\\
&\quad>L^2 T^{-1}\bigl\anglel y(t),u-v\bigr\angler^2.
\end{aligned} %
\]
Thus for any $\delta>0$ and $T>T_0$
\[
\varPhi_T(u,v)\ge L^2\bigl\anglel J_T(u-v),
\,u-v\bigr\angler >L^2 \bigl(\lambda_{\min
}(J)-\delta \bigr)
\|u-v\|^2,
\]
and we have obtained the left-hand side of \eqref{inq4int_dif_H^2} with
the constant $c_0<L^2\lambda_{\min}(J)$, where $\lambda_{\min}(J)$ is
the least eigenvalue of the positive definite matrix $J$.

The next fact is a reformulation of Corollary \ref{cor_sPLD_St_sGN} for
the regression function $a(t,\,\tau)$ of our example.

\begin{core}
Under conditions {\bf N1}, {\bf N4(i)}, {\bf N4(ii)} and \eqref{cnv_J_T_2_J} there
exist constants $B$, $b>0$ such that for $T>T_0$, $R>R_0$
\[
P \bigl\{ \bigl\llVert T^{\frac{1}2} (\theta_T- \theta ) \bigr
\rrVert \ge R \bigr\} \le B\exp \bigl\{-b R^2 \bigr\}.
\]
Moreover for any $\beta>0$ the constant $B$ can be chosen so that
\[
b\ge\frac{L^2\lambda_{\min}(J)} {16\pi f_0(1+q)}-\beta.
\]
\end{core}
\end{exm}

\begin{exm}
Here we offer an example of the j.s.s.-G. stochastic process $\xi$
with orthogonal increments in the formula \eqref{int_repr_epsln} using
the Ito--Nicio series (see \cite{ItoNi_otCoSoIBSVRV} and references therein).

Consider any orthonormal basis $ \{\varphi_k,\,k\ge1 \}$\ in\ $L_2 (\mathbb{R}_+ )$ and a sequence $\{Z_k,\break k\ge
1\}$ of independent $N(0,1)$ r.v.'s. Then
\[
w_0(t)=\sum\limits
_{k=1}^\infty
\,Z_k\int\limits
_0^t\,\varphi_k(u)du,\quad t
\ge0,
\]
is a standard Wiener process with covariances $Ew_0(t)w_0(s)=\min\{t,s\}
$. We need some kind of the Wiener process on the real line $\mathbb
{R}$. Let $ \{w_1(t),\,t\ge0 \}$, $ \{w_2(t),\,t\ge0
\}$ be two independent Wiener processes
of the following form:
\[
w_i(t)=\sum\limits
_{k=1}^\infty
\,Z_{ik}\int\limits
_0^t\,\varphi _k(u)du,\quad t
\ge0,\ i=1,2,
\]
and $ \{Z_{ik},\,k\ge1,\,i=1,2 \}$ be independent $N(0,1)$
r.v.'s. Then the required Wiener process on $\mathbb{R}$ can be defined
as $w(t)=w_1(t)$, $t\ge0$, and $w(t)=w_2(|t|)$, $t<0$. For any real
$t_1<t_2\le t_3<t_4$
\begin{equation}
E \bigl(w(t_2)-w(t_1) \bigr) \bigl(w(t_4)-w(t_3)
\bigr)=0, \label{indep_incr_exle}
\end{equation}
i.e. increments are orthogonal. On the other hand, for any $t>s$
\[
E \bigl(w(t)-w(s) \bigr)^2=t-s.
\]

Let $ \{\xi_{ik},\,k\ge1,\,i=1,2 \}$ be i.i.d. s.s.-G. r.v.'s
(and non-Gaussian!) with unit variance. Some examples of s.s.-G. r.v.'s
can be found in \cite{BuKo_MCoRVRP}. Thus the Bernoulli r.v. and the
r.v. uniformly distributed in $ [-\sqrt{3},\,\sqrt{3} ]$ are
s.s.-G. and have unit variances.

Let us introduce the stochastic processes
\[
\xi_i(t)=\sum\limits
_{k=1}^\infty\,
\xi_{ik}\int\limits
_0^t\,\varphi _k(u)du,\quad t
\ge0,\ i=1,2,
\]
$\xi(t)=\xi_1(t)$, $t\ge0$, and $\xi(t)=\xi_2(|t|)$, $t<0$. Then $\xi
= \{\xi(t),\,t\in\mathbb{R} \}$ is a process with orthogonal
increments and is not a Gaussian one.

However, it is a j.s.s.-G. process. To prove this statement consider
arbitrary numbers $t_1<\cdots<t_n$, where the first $m$ numbers, $0\le
m\le n$, are negative and the rest $n-m$ numbers are positive. Let
$\varDelta= (\varDelta_1,\,\ldots,\,\varDelta_n )'\in\mathbb{R}^n$ be
any vector. Then
\[
\begin{aligned} \varSigma_2&=\sum
\limits_{k=1}^N
\,\xi_{2k} \Biggl(\sum\limits
_{i=1}^m\,
\varDelta_i\int\limits
_0^{ \llvert t_i \rrvert }\,\varphi_k(u)du
\Biggr)\to\sum\limits
_{i=1}^m\,\varDelta_i
\xi_2 \bigl( \llvert t_i \rrvert \bigr)\ \text {a.s.},\quad
\text{as}\ N\to\infty,
\\
\varSigma_1&=\sum\limits
_{k=1}^N\,
\xi_{1k} \Biggl(\sum\limits
_{i=m+1}^n\,
\varDelta_i\int\limits
_0^{t_i}\,\varphi_k(u)du
\Biggr)\to\sum\limits
_{i=m+1}^n\,\varDelta_i
\xi_1 (t_i )\ \text{a.s.},\quad \text{as}\ N\to \infty.
\end{aligned} %
\]
By the Fatou lemma
%
\begin{align*}
E\exp \Biggl\{\sum\limits
_{i=1}^n\,\varDelta_i
\xi(t_i) \Biggr\}&=E\underset {N\to\infty}\lim\,\exp \{
\varSigma_2+\varSigma_1 \}
\\
&\le\underset{N\to\infty}\liminf\,\prod\limits
_{k=1}^N
\, E\exp \Biggl\{ \xi_{2k} \Biggl(\sum\limits
_{i=1}^m
\,\varDelta_i\int\limits
_0^{ \llvert t_i \rrvert }\,\varphi_k(u)du
\Biggr) \Biggr\}\cdot
\\
&\qquad\cdot \prod\limits
_{k=1}^N\, E\exp \Biggl
\{\xi_{1k} \Biggl(\sum\limits
_{i=m+1}^n\,
\varDelta_i\int\limits
_0^{t_i}\,\varphi_k(u)du
\Biggr) \Biggr\}\le
\\
&\quad\le\underset{N\to\infty}\lim\,\exp \Biggl\{\frac{1}2 \Biggl(\sum
\limits
_{k=1}^N\, \Biggl(\sum\limits
_{i=1}^m
\,\varDelta_i\int\limits
_0^{ \llvert t_i \rrvert }\,\varphi_k(u)du
\Biggr)^2
\\
&\qquad+\sum\limits
_{k=1}^N\,
\Biggl(\sum\limits
_{i=m+1}^n\, \varDelta_i\int
\limits_0^{t_i}
\, \varphi_k(u)du \Biggr)^2 \Biggr) \Biggr\}\\
%
&\quad=\underset{N\to\infty}
\lim\,\exp \Biggl\{\frac{1}2 \Biggl(\sum\limits
_{i,j=1}^m
\, \Biggl(\sum\limits
_{k=1}^N\,\int\limits
_0^{ \llvert t_i \rrvert }
\,\varphi_k(u)du \int\limits
_0^{ \llvert t_j \rrvert }\,\varphi
_k(u)du \Biggr)\varDelta_i\varDelta_j
\\
&\qquad+\sum\limits
_{i,j=m+1}^n\,
\Biggl(\sum\limits
_{k=1}^N\,\int\limits
_0^{t_i}\,
\varphi_k(u)du \int\limits
_0^{t_j}\, \varphi_k(u)du
\Biggr)\varDelta_i\varDelta_j \Biggr) \Biggr\}.
\end{align*}
%
By Parseval's identity
\[
\begin{aligned} &\underset{N\to\infty}\lim\,\sum
\limits_{k=1}^N
\,\int\limits
_0^{ \llvert t_i \rrvert }\,\varphi_k(u)du \int\limits
_0^{ \llvert t_j \rrvert }
\,\varphi _k(u)du
\\
&\quad=\sum\limits
_{k=1}^\infty\,\int\limits
_0^\infty\,
\chi_{
[0,\, \llvert t_i \rrvert  ]}(u)\varphi_k(u)du \int\limits
_0^\infty\,
\chi_{ [0,\, \llvert t_j \rrvert  ]}(u)\varphi_k(u)du
\\
&\quad=\int\limits
_0^\infty\,\chi_{ [0,\, \llvert t_i \rrvert  ]}(u)\chi
_{ [0,\, \llvert t_j \rrvert  ]}(u)du=\min \bigl\{ \llvert t_i \rrvert ,\, \llvert
t_j \rrvert \bigr\}. \end{aligned} %
\]
Similarly
\[
\underset{N\to\infty}\lim\,\sum\limits
_{k=1}^N\,\int
\limits_0^{t_i}
\, \varphi_k(u)du \int\limits
_0^{t_j}\,\varphi_k(u)du=
\min \{t_i,\, t_j \}.
\]
It means that
\[
E\exp \Biggl\{\sum\limits
_{i=1}^n\,\varDelta_i
\xi(t_i) \Biggr\}\le\exp \biggl\{ \frac{1}2\anglel B
\varDelta,\,\varDelta\angler \biggr\}
\]
with
\[
B=\lleft( %
\begin{array}{cc} B_2 & 0 \\ 0 & B_1
\end{array} %
 \rright),
\]
where $B_2= (\min \{ \llvert t_i \rrvert ,\,  \llvert t_j \rrvert  \}
 )_{i,j=1}^m$, $B_1= (\min \{t_i,\,t_j \}
)_{i,j=m+1}^n$, and the process $\xi$ is j.s.s.-G.
\end{exm}





\end{document}